\documentclass[10pt, leqno]{amsart}
\usepackage{amsfonts,amssymb}

\usepackage{amsmath}
\usepackage{amsthm}
\usepackage{amsrefs}
\usepackage{qsymbols}
\usepackage{latexsym}
\usepackage{chngcntr}
\usepackage[noadjust]{cite}
\usepackage{paralist}
\usepackage{esint}

\newtheorem{theorem}{Theorem}[section]

\newtheorem{lemma}[theorem]{Lemma}
\newtheorem{proposition}[theorem]{Proposition}
\newtheorem{corollary}[theorem]{Corollary}
\theoremstyle{definition}

\newtheorem{remark}[theorem]{Remark}

\newcommand{\IR}{\mathbb{R}}
\newcommand{\IC}{\mathbb{C}}
\newcommand{\IN}{\mathbb{N}}

\newcommand{\IP}{\mathbb{P}}

\newcommand{\IF}{\mathbb{F}}
\newcommand{\IE}{\mathbb{E}}

\newcommand{\Lop}{{\mathcal{L}}}

\newcommand{\Sec}{\mathrm{S}}

\newcommand{\cT}{\mathcal{T}}
\newcommand{\cR}{\mathcal{R}}

\renewcommand{\L}{\mathrm{L}}
\renewcommand{\H}{\mathrm{H}}

\newcommand{\B}{\mathrm{B}}
\newcommand{\C}{\mathrm{C}}
\newcommand{\W}{\mathrm{W}}

\renewcommand{\d}{\mathrm{d}}

\newcommand{\e}{\mathrm{e}}
\newcommand{\eps}{\varepsilon}

\DeclareMathOperator{\divergence}{div}

\DeclareMathOperator{\Id}{Id}

\newcommand{\dom}{\mathcal{D}}

\newcommand{\fa}{\mathfrak{a}}
\newcommand{\NSE}{\mathrm{(NSE)}}

\numberwithin{equation}{section}

\title[$\L^p$-theory of the Navier--Stokes equations on bounded Lipschitz domains]{On the $\L^p$-theory of the Navier--Stokes equations on three-dimensional bounded Lipschitz domains}
\author{Patrick Tolksdorf}
\address{Fachbereich Mathematik, Technische Universit\"at Darmstadt, Schlossgartenstr. 7, 64289 Darmstadt, Germany}
\email{tolksdorf@mathematik.tu-darmstadt.de}
\thanks{The author was supported by ``Studienstiftung des deutschen Volkes''.}

\date{\today}

\begin{document}
\begin{abstract}
On a bounded Lipschitz domain $\Omega \subset \IR^d$, $d \geq 3$, we continue the study of Shen~\cite{Shen-Stokes} and of Kunstmann and Weis~\cite{Kunstmann_Weis} of the Stokes operator on $\L^p_{\sigma} (\Omega)$. We employ their results in order to determine the domain of the square root of the Stokes operator as the space $\W^{1 , p}_{0 , \sigma} (\Omega)$ for $\lvert \frac{1}{p} - \frac{1}{2} \rvert < \frac{1}{2d} + \eps$ and some $\eps > 0$. This characterization provides gradient estimates as well as $\L^p$-$\L^q$-mapping properties of the corresponding semigroup. In the three-dimensional case this provides a means to show the existence of solutions to the Navier--Stokes equations in the critical space $\L^{\infty} (0 , \infty ; \L^3_{\sigma} (\Omega))$ whenever the initial velocity is small in the $\L^3$-norm. Finally, we present a different approach to the $\L^p$-theory of the Navier--Stokes equations by employing the maximal regularity proven by Kunstmann and Weis~\cite{Kunstmann_Weis}.
\end{abstract}
\maketitle

\section{Introduction}
\label{Sec: Introduction}

\noindent In this article we consider the incompressible Navier--Stokes equations
\begin{align*}
\NSE \quad \left\{
\begin{aligned}
 \partial_t u - \Delta u + (u \cdot \nabla) u + \nabla \pi &= f \quad &&\text{in } \Omega , t > 0 \\
 \divergence(u) &= 0 &&\text{in } \Omega , t > 0 \\
 u(0) &= a &&\text{in } \Omega \\
 u &= 0 &&\text{on } \partial \Omega , t > 0,
\end{aligned}
\right.
\end{align*}
on a bounded Lipschitz domain $\Omega \subset \IR^3$. Here, $u$ denotes a vector field which corresponds to the velocity of an incompressible fluid that governs $\Omega$, $\pi$ the pressure inside $\Omega$, $f$ an external force, and $a$ the initial velocity. While there exists extensive literature to this equation if $\Omega$ is smooth, see, e.g.,~\cite{Ladyzhenskaya, Sohr, Temam}, the investigation for bounded Lipschitz domains started fairly recently. For example one of the first existence results of strong solutions was given by Deuring and von Wahl~\cite{Deuring_von-Wahl} in 1995 and was ultimately improved by Mitrea and Monniaux~\cite{Mitrea_Monniaux}. In both articles the authors followed the classical approach of Fujita and Kato~\cite{Fujita_Kato}. Another existence result --- proven in a very short and elegant way --- was given by Taylor~\cite{Taylor}. However, it has to be noted, that all approaches use only $\L^2$-theory in order to establish the existence theorems. \par
It is well-known, that also the $\L^p$-theory is of great interest, for example for uniqueness and regularity questions. However, an $\L^p$-theory seemed to be out of reach for a long time as it was not even known that the Stokes operator generates a strongly continuous semigroup on $\L^p_{\sigma} (\Omega)$ for $p \neq 2$. Due to the boundedness properties of the Helmholtz projection, see~\cite{Fabes_Mendez_Mitrea}, Taylor conjectured in~\cite{Taylor}, that for each bounded Lipschitz domain, there exists $\eps > 0$ such that the Stokes operator generates an analytic semigroup on $\L^p_{\sigma} (\Omega)$ whenever $3 / 2 - \eps < p < 3 + \eps$. Taylor's conjecture was answered to the affirmative by Shen in the seminal paper~\cite{Shen-Stokes}. \par
To establish an $\L^p$-theory for the Navier--Stokes equations there are several ways one can take and we will present two of them in this article. The first is the classical approach to obtain mild solutions starting from $\L^3$-initial data via an iteration scheme, which was first performed by Giga and Miyakawa~\cite{Giga_Miyakawa} using fractional powers of the Stokes operator for bounded and smooth domains. The theory was extended to the whole space by Kato~\cite{Kato_iteration} and the approach was adjusted by Giga~\cite{Giga_iteration} for bounded and smooth domains. This approach requires certain $\L^p$-$\L^q$-mapping properties as well as gradient estimates of the Stokes semigroup. These estimates will be established in this work by using the boundedness of the $\H^{\infty}$-calculus of the Stokes operator proven by Kunstmann and Weis~\cite{Kunstmann_Weis}. \par
The second approach uses the maximal $\L^q$-regularity of the Stokes operator. This approach was initiated by Solonnikov~\cite{Solonnikov} and largely improved in the subsequent years, see, e.g.,~\cite{Amann , DHP , Geissert_et_al}. The property that the Stokes operator on $\L^p_{\sigma} (\Omega)$ has maximal $\L^q$-regularity was recently proven by Kunstmann and Weis~\cite{Kunstmann_Weis} for the same range of numbers $p$ as the analytic semigroup exists. Independently, this property was proven in the PhD-thesis of the author, see~\cite{Tolksdorf}. \par
For a further historical review the reader may consult the introductions in~\cite{Deuring_von-Wahl , Taylor , Mitrea_Monniaux}. \par
In the remainder of this introduction, we formulate the main results presented in this article. We refer to Section~\ref{Sec: Notation} for the respective notation. Note that the results dealing with the linear theory are formulated in $\IR^d$ with $d \geq 3$.

\begin{theorem}
\label{Thm: Square root}
Let $\Omega \subset \IR^d$, $d \geq 3$, be a bounded Lipschitz domain. Then there exists $\eps > 0$ depending only on $d$ and the Lipschitz character of $\Omega$ such that for all
\begin{align}
\label{Eq: Admissible interval of p's}
 \Big\lvert \frac{1}{p} - \frac{1}{2} \Big\rvert < \frac{1}{2 d} + \eps 
\end{align}
the domain of the square root of the Stokes operator $A_p$ on $\L^p_{\sigma} (\Omega)$ coincides with $\W^{1 , p}_{0 , \sigma} (\Omega)$, i.e.,
\begin{align*}
 \dom(A_p^{1 / 2}) = \W^{1 , p}_{0 , \sigma} (\Omega)
\end{align*}
with equivalence of the respective norms.
\end{theorem}
A corollary are the $\L^p$-$\L^q$-mapping properties and gradient estimates of the Stokes semigroup.

\begin{corollary}
\label{Cor: Lp-Lq}
Let $\Omega \subset \IR^d$, $d \geq 3$, be a bounded Lipschitz domain. Then there exists $\eps > 0$ depending only on $d$ and the Lipschitz character of $\Omega$ such that for all $p \leq q$ that satisfy
\begin{align*}
 \Big\lvert \frac{1}{p} - \frac{1}{2} \Big\rvert < \frac{1}{2 d} + \eps \quad \text{and} \quad \Big\lvert \frac{1}{q} - \frac{1}{2} \Big\rvert < \frac{1}{2 d} + \eps
\end{align*}
the Stokes semigroup satisfies the estimates
\begin{align*}
 \| \e^{- t A_p} f \|_{\L^q_{\sigma} (\Omega)} \leq C t^{- \frac{d}{2} (\frac{1}{p} - \frac{1}{q})} \| f \|_{\L^p_{\sigma} (\Omega)} \qquad (f \in \L^p_{\sigma} (\Omega))
\end{align*}
and
\begin{align*}
 \| \nabla \e^{- t A_p} f \|_{\L^p(\Omega ; \IC^{d^2})} \leq C t^{- \frac{1}{2}} \| f \|_{\L^p_{\sigma} (\Omega)} \qquad (f \in \L^p_{\sigma} (\Omega)).
\end{align*}
Here, the constants $C$ are independent of $t$ and $f$.
\end{corollary}

As we mentioned before, for the $\L^p$-theory of the Navier--Stokes equations we pursue two approaches, delivering two types of theorems. The first is derived by following the classical approach of Giga~\cite{Giga_iteration} and Kato~\cite{Kato_iteration}. Note that we take $f = 0$ in $\NSE$ in this theorem.

\begin{theorem}
\label{Thm: Critical space}
Let $\Omega \subset \IR^3$ be a bounded Lipschitz domain. Then there exists $\eps > 0$ depending only on $d$ and the Lipschitz character of $\Omega$ such that for all $3 \leq r < 3 + \eps$ and all $a \in \L^r_{\sigma} (\Omega)$ the following statements are valid.
\begin{enumerate}
 \item There exists $T_0 > 0$ and a mild solution $u : [0 , T_0) \to \L^r_{\sigma} (\Omega)$ to $\NSE$ with $f = 0$ and initial velocity $a$ that satisfies for all $r \leq p < 3 + \eps$ with $1 / r - 1 / p < 1 / 6$
 \begin{align*}
  t \mapsto t^{\frac{3}{2} (\frac{1}{r} - \frac{1}{p})} u (t) &\in \B\C ([0 , T_0) ; \L^p_{\sigma} (\Omega)), \\
  t \mapsto t^{\frac{1}{2} + \frac{3}{2} (\frac{1}{r} - \frac{1}{p})} \nabla u(t) &\in \B \C ([0 , T_0) ; \L^p (\Omega ; \IC^9)).
 \end{align*}
Moreover,
\begin{align*}
 \| u(t) - a \|_{\L^r_{\sigma} (\Omega)} &\to 0 \quad \text{as} \quad t \searrow 0, \\
  t^{\frac{1}{2} + \frac{3}{2} (\frac{1}{r} - \frac{1}{p})} \| \nabla u(t) \|_{\L^p(\Omega ; \IC^9)} &\to 0 \quad \text{as} \quad t \searrow 0,
\end{align*}
and if $r < p < 3 + \eps$, then
\begin{align*}
 t^{\frac{3}{2} (\frac{1}{r} - \frac{1}{p})} \|u(t)\|_{\L^p_{\sigma} (\Omega)} &\to 0 \quad \text{as} \quad t \searrow 0.
\end{align*}
 \item If $r > 3$, there exists a constant $C > 0$ depending only on $r$, $p$, and the constants in Corollary~\ref{Cor: Lp-Lq}, such that
 \begin{align*}
  T_0 \geq C \|a\|^{- \frac{2 r}{r - 3}}_{\L^r_{\sigma} (\Omega)}.
 \end{align*}
 \item For all $3 \leq p < 3 + \eps$ there are positive constants $C_1 , C_2 > 0$ depending only on $p$ and the constants in Corollary~\ref{Cor: Lp-Lq} such that if $\|a\|_{\L^3_{\sigma} (\Omega)} < C_1$ the mild solution is global, i.e., $T_0 = \infty$. Moreover, this solution satisfies the estimates
 \begin{align*}
 \begin{aligned}
  \|u(t)\|_{\L^p_{\sigma} (\Omega)} &\leq C_2 t^{\frac{3}{2 p} - \frac{1}{2}} \quad &&(0 < t < \infty) \\
  \| \nabla u (t) \|_{\L^p (\Omega ; \IC^9)} &\leq C_2 t^{\frac{3}{2 p} - 1} &&(0 < t < \infty).
 \end{aligned}
 \end{align*}
\end{enumerate}
\end{theorem}
The final result uses the approach via maximal $\L^q$-regularity of the Stokes operator in order to obtain strong solutions to the Navier--Stokes equations.

\begin{theorem}
\label{Thm: Solutions via maximal regularity}
Let $\Omega \subset \IR^3$ be a bounded Lipschitz domain. Then there exists $\eps > 0$ depending only on $d$ and the Lipschitz character of $\Omega$ such that for all $2 \leq p < 3 + \eps$ and all numbers $q$ satisfying $2 \leq q < \infty$ if $p = 2$ and
\begin{align*}
 \frac{2 (p + 1)}{3} < q < \infty 
\end{align*}
if $2 < p < 3 + \eps$ the following statement is valid: There exists $C > 0$ such that for all initial velocities $a$ in the real interpolation space $(\L^p_{\sigma} (\Omega) , \dom(A_p))_{1 - 1 / q , q}$ and all $f \in \L^q(0 , \infty ; \L^p (\Omega ; \IC^3))$ with
\begin{align*}
 \| a \|_{(\L^p_{\sigma} (\Omega) , \dom(A_p))_{1 - 1 / q , q}} + \| f \|_{\L^q(0 , \infty ; \L^p (\Omega ; \IC^3))} < C
\end{align*}
there exists a strong solution $u$ to $\NSE$ in the space
\begin{align*}
 \W^{1 , q}(0 , \infty ; \L^p_{\sigma} (\Omega)) \cap \L^q(0 , \infty ; \dom(A_p)).
\end{align*}
\end{theorem}

If $\Omega$ is bounded and smooth, then Theorem~\ref{Thm: Square root} is known due to Giga~\cite{Giga_fractional} and Corollary~\ref{Cor: Lp-Lq} and Theorem~\ref{Thm: Critical space} are known due to Giga~\cite{Giga_iteration}. Finally, for a theorem in the fashion of Theorem~\ref{Thm: Solutions via maximal regularity} see Amann~\cite{Amann} or Giga and Sohr~\cite{Giga_Sohr}. \par
This article is organized as follows. In Section~\ref{Sec: Notation} we provide the required notation as well as some important facts about the Stokes operator on bounded Lipschitz domains. Section~\ref{Sec: Square root} is concerned with the proofs of Theorem~\ref{Thm: Square root} and Corollary~\ref{Cor: Lp-Lq}. The final Section~\ref{Sec: Nonlinear} is split into two parts. Here, the first part deals with Theorem~\ref{Thm: Critical space} and the second with Theorem~\ref{Thm: Solutions via maximal regularity}. A precise definition of the notion ``solution'' is given in the respective subsections.

\subsection*{Acknowledgments}
I would like to thank Robert Haller-Dintelmann for the supervision and the support during the time of my PhD-studies.

\section{Notation and preliminary results}
\label{Sec: Notation}

First, we fix some notation. In the whole article, the space dimension of the underlying Euclidean space will be $d \geq 3$. An open set $\Omega \subset \IR^d$ will be called a bounded Lipschitz domain if the boundary can locally be expressed as the graph of a Lipschitz continuous function. The domain of a linear operator $A$ on a Banach space $X$ is denoted by $\dom(A)$ and for an interval $I \subset \IR$ we write $\B\C(I ; X)$ for all bounded and continuous functions with values in $X$. \par
For an open and bounded set $\Omega \subset \IR^d$ and $1 < p < \infty$ the $\L^p$-space of solenoidal vector fields $\L^p_{\sigma} (\Omega)$ is defined as the closure of $\C_{c , \sigma}^{\infty} (\Omega)$ in $\L^p(\Omega ; \IC^d)$, where
\begin{align*}
 \C_{c , \sigma}^{\infty} (\Omega) := \{ \varphi \in \C_c^{\infty} (\Omega ; \IC^d) : \divergence(\varphi) = 0 \}.
\end{align*}
The first-order Sobolev space of solenoidal vector fields $\W^{1 , p}_{0 , \sigma} (\Omega)$ is defined as the closure of $\C_{c , \sigma}^{\infty} (\Omega)$ in $\W^{1 , p} (\Omega ; \IC^d)$ and the space of $\L^p$-integrable gradient fields is defined by
\begin{align*}
 \nabla \W^{1 , p} (\Omega) := \{ \nabla u : u \in \W^{1 , p} (\Omega) \}.
\end{align*}
Because $\L^2_{\sigma} (\Omega)$ is a closed subspace of $\L^2(\Omega ; \IC^d)$ it is clear that the orthogonal projection $\IP_2$ from $\L^2(\Omega ; \IC^d)$ onto $\L^2_{\sigma} (\Omega)$ exists and is bounded. This projection is called Helmholtz projection. The boundedness of the Helmholtz projection on $\L^p$-spaces for $p$ in an open interval about two is a well-known result of Fabes, Mendez, and Mitrea~\cite[Thm.~11.1]{Fabes_Mendez_Mitrea} and is stated in the following theorem.

\begin{theorem}[Fabes, Mendez \& Mitrea]
\label{Thm: Fabes_Mendez_Mitrea}
Let $\Omega \subset \IR^d$ be a bounded Lipschitz domain. There exists $\eps > 0$ such that for all $3 / 2 - \eps < p < 3 + \eps$ the Helmholtz projection restricts/extends to a bounded projection $\IP_p$ on $\L^p(\Omega ; \IC^d)$ with range $\L^p_{\sigma} (\Omega)$. Moreover, the range of $\Id - \IP_p$ is given by $\nabla \W^{1 , p} (\Omega)$.
\end{theorem}

The Stokes operator $A_2$ on an open and bounded set $\Omega \subset \IR^d$ is defined by means of Kato's form method as the $\L^2_{\sigma} (\Omega)$-realization of the sesquilinear form
\begin{align*}
 \fa : \W^{1 , 2}_{0 , \sigma} (\Omega) \times \W^{1 , 2}_{0 , \sigma} (\Omega) \to \IC , \quad (u , v) \mapsto \int_{\Omega} \nabla u \cdot \overline{\nabla v} \; \d x.
\end{align*}
By symmetry and coercivity of the form it is clear that $A_2$ is self-adjoint and that $- A_2$ generates a bounded analytic semigroup. This semigroup is denoted by $\e^{- t A_2}$ and is called the Stokes semigroup. For a clear discussion of the facts above, see~\cite[Sec.~4]{Mitrea_Monniaux}. In particular, Mitrea and Monniaux give in~\cite[Thm.~4.7]{Mitrea_Monniaux} the following convenient characterization of the Stokes operator.

\begin{theorem}[Mitrea \& Monniaux]
\label{Thm: Characterization Stokes}
If $\Omega \subset \IR^d$ is a bounded Lipschitz domain, then the Stokes operator on $\L^2_{\sigma} (\Omega)$ is characterized by
\begin{align*}
 \dom(A_2) &= \{ u \in \W^{1 , 2}_{0 , \sigma} (\Omega) : \exists \pi \in \L^2(\Omega) \text{ such that } - \Delta u + \nabla \pi \in \L^2_{\sigma} (\Omega) \} \\
 A_2 u &= - \Delta u + \nabla \pi.
\end{align*}
\end{theorem}

Note that in the theorem above, ``$- \Delta u + \nabla \pi$'' has to be understood in the sense of distributions. \par
To define the Stokes operator on the spaces $\L^p_{\sigma} (\Omega)$, one distinguishes the cases $p> 2$ and $p < 2$. If $p > 2$, then the Stokes operator $A_p$ on $\L^p_{\sigma} (\Omega)$ is defined as the part of $A_2$ in $\L^p_{\sigma} (\Omega)$, i.e.,
\begin{align*}
 \dom(A_p) := \{ u \in \dom(A_2) : A_2 u \in \L^p_{\sigma} (\Omega) \}, \quad A_p u := A_2 u.
\end{align*}
If $p < 2$ and if $A_2$ is closable in $\L^p_{\sigma} (\Omega)$, then the Stokes operator $A_p$ is defined as the closure of $A_2$ in $\L^p_{\sigma} (\Omega)$, i.e.,
\begin{align*}
 \dom(A_p) &:= \{ u \in \L^p_{\sigma} (\Omega) : \exists f \in \L^p_{\sigma} (\Omega) ,~ \exists (u_n)_{n \in \IN} \subset \dom(A_2) \text{ such that } u_n \to u \text{ and } A_2 u_n \to f \text{ in } \L^p\} \\
 A_p u &:= f.
\end{align*}

\begin{remark}
If $2 < p < \infty$ is such that the Helmholtz projection is bounded on $\L^p (\Omega ; \IC^d)$, then~\cite[Prop.~5.2.16]{Tolksdorf} implies that $A_p$ is densely defined if and only if $A_2$ is closable in $\L^{p^{\prime}}_{\sigma} (\Omega)$, where $1 / p + 1 / p^{\prime} = 1$. If this applies, then $A_p^* = A_{p^{\prime}}$, where $A_p^*$ denotes the dual operator to $A_p$. 
\end{remark}

In 2012, Shen proved in his seminal paper~\cite{Shen-Stokes} the following result.

\begin{theorem}[Shen]
\label{Thm: Shen}
Let $\Omega \subset \IR^d$ be a bounded Lipschitz domain. Then there exists $\eps > 0$ depending only on $d$ and the Lipschitz character of $\Omega$ such that for all
\begin{align*}
 \Big\lvert \frac{1}{p} - \frac{1}{2} \Big\rvert < \frac{1}{2 d} + \eps 
\end{align*}
$A_p$ is sectorial of angle $0$. In particular, $- A_p$ generates a bounded analytic semigroup with $0 \in \rho(A_p)$ and $A_p$ is closed and densely defined.
\end{theorem}

Let us quantify Shen's statement that the Stokes operator is densely defined. For this purpose, define for $1 < p < \infty$ the space $\W^{2 , p}_{0 , \sigma} (\Omega)$ as the closure of $\C_{c , \sigma}^{\infty} (\Omega)$ in $\W^{2 , p} (\Omega ; \IC^d)$. Then, the following lemma is valid.

\begin{lemma}
\label{Lem: Embedding}
Let $\Omega \subset \IR^d$ be a bounded Lipschitz domain. Then there exists $\eps > 0$ such that for all
\begin{align*}
 \Big\lvert \frac{1}{p} - \frac{1}{2} \Big\rvert < \frac{1}{2 d} + \eps ,
\end{align*}
the space $\W^{2 , p}_{0 , \sigma} (\Omega)$ is embedded continuously into $\dom(A_p)$. In particular, the representation formula
\begin{align}
\label{Eq: Stokes for smooth functions}
 A_p u = - \IP_p \Delta u \qquad (u \in \W^{2 , p}_{0 , \sigma} (\Omega))
\end{align}
is valid.
\end{lemma}

\begin{proof}
Let $u \in \W^{2 , p}_{0 , \sigma} (\Omega)$. We distinguish the cases $p \geq 2$ and $p < 2$. If $p \geq 2$, then, by virtue of Theorem~\ref{Thm: Fabes_Mendez_Mitrea}, there exists $g \in \W^{1 , p} (\Omega)$ such that $(\Id - \IP_p) \Delta u = \nabla g$. Consequently,
\begin{align*}
 - \Delta u = - \IP_p \Delta u - (\Id - \IP_p) \Delta u = - \IP_p \Delta u - \nabla g.
\end{align*}
Since $\Omega$ is bounded, Theorem~\ref{Thm: Characterization Stokes} gives $u \in \dom(A_2)$. The definition of $\dom(A_p)$ then delivers $u \in \dom(A_p)$. In particular,~\eqref{Eq: Stokes for smooth functions} is valid. \par
If $p < 2$, let $(u_n)_{n \in \IN} \subset \C_{c , \sigma}^{\infty} (\Omega)$ be an appropriate sequence that approximates $u$ in $\W^{2 , p} (\Omega ; \IC^d)$. Since $\IP_2$ extends to a bounded operator on $\L^p (\Omega ; \IC^d)$ by Theorem~\ref{Thm: Fabes_Mendez_Mitrea}, the representation formula~\eqref{Eq: Stokes for smooth functions} for $A_2$ shows that $(A_2 u_n)_{n \in \IN}$ is a Cauchy sequence in $\L^p_{\sigma} (\Omega)$. Since $A_p$ is the closure of $A_2$ in $\L^p_{\sigma} (\Omega)$ this proves $u \in \dom(A_p)$ together with~\eqref{Eq: Stokes for smooth functions}. The continuous embedding follows by the boundedness of $\IP_p$ and the validity of~\eqref{Eq: Stokes for smooth functions}.
\end{proof}

We close this section by mentioning some functional analytic facts of the Stokes operator on $\L^p_{\sigma} (\Omega)$. We start by introducing the notion of maximal regularity. \par
For $1 < p , q < \infty$ define the maximal regularity space $\IE$ with corresponding data space $\IF$ by
\begin{align}
\label{Eq: Spaces of maximal regularity}
 \IE := \W^{1 , q} (0 , \infty ; \L^p_{\sigma} (\Omega)) \cap \L^q(0 , \infty ; \dom(A_p)), \quad \IF := \L^q(0 , \infty ; \L^p_{\sigma} (\Omega)) \times (\L^p_{\sigma} (\Omega) , \dom(A_p))_{1 - 1 / q , q}
\end{align}
endowed with the canonical norms. Here, $(\cdot , \cdot)_{1 - 1 / q , q}$ denotes the real interpolation functor. We say that the Stokes operator has \textit{maximal $\L^q$-regularity} if the operator $\partial_t + A_p$ is an isomorphism between $\IE$ and $\IF$. The following theorem concerns the maximal $\L^q$-regularity of the Stokes operator on a bounded Lipschitz domain and can be found in~\cite[Prop.~13]{Kunstmann_Weis} or~\cite[Thm.~5.2.24]{Tolksdorf}.

\begin{theorem}[Kunstmann \& Weis]
\label{Thm: Maximal regularity}
Let $\Omega \subset \IR^d$ be a bounded Lipschitz domain and $1 < q < \infty$. Then there exists $\eps > 0$ depending only on $d$ and the Lipschitz character of $\Omega$ such that for all
\begin{align*}
 \Big\lvert \frac{1}{p} - \frac{1}{2} \Big\rvert < \frac{1}{2 d} + \eps 
\end{align*}
the Stokes operator $A_p$ has maximal $\L^q$-regularity.
\end{theorem}

Since $A_p$ is injective and sectorial of angle $0$, one can follow the construction in Haase~\cite[Ch.~2]{Haase} to assign a linear and closed operator $f(A_p)$ to each holomorphic function $f : \Sec_{\theta} \to \IC$ exhibiting at most polynomial growth at $0$ and infinity, where for $\theta \in (0 , \pi)$
\begin{align*}
 \Sec_{\theta} := \{ z \in \IC \setminus \{0\} : \lvert \arg(z) \rvert < \theta \}.
\end{align*}
In particular, one can assign an operator to functions $f \in \H^{\infty} (\Sec_{\theta})$, which is the algebra of bounded and holomorphic functions on $\Sec_{\theta}$, and to functions of the form $\Sec_{\theta} \ni z \mapsto z^{\alpha}$ for each $\alpha \in \IC$. The latter type of functions lead to fractional powers $A_p^{\alpha}$ of the Stokes operator. If for each $f \in \H^{\infty} (\Sec_{\theta})$ the operator $f(A_p)$ is bounded and if one has the estimate
\begin{align*}
 \| f(A_p) \|_{\Lop(\L^p_{\sigma} (\Omega))} \leq C \| f \|_{\L^{\infty} (\Sec_{\theta})} \qquad (f \in \H^{\infty} (\Sec_{\theta})),
\end{align*}
then one says that the $\H^{\infty}$-calculus of $A_p$ is bounded. This is exactly what Kunstmann and Weis proved in~\cite[Thm.~16]{Kunstmann_Weis}.

\begin{theorem}[Kunstmann \& Weis]
Let $\Omega \subset \IR^d$ be a bounded Lipschitz domain. Then there exists $\eps > 0$ depending only on $d$ and the Lipschitz character of $\Omega$ such that for all
\begin{align*}
 \Big\lvert \frac{1}{p} - \frac{1}{2} \Big\rvert < \frac{1}{2 d} + \eps 
\end{align*}
the $\H^{\infty}$-calculus of the Stokes operator $A_p$ is bounded.
\end{theorem}

The boundedness of the $\H^{\infty}$-calculus of $A_p$ gives information about the domains of the fractional powers of $A_p$. Indeed, if $p$ is in the same range as in the preceding theorem, these can be computed by complex interpolation
\begin{align}
\label{Eq: Complex interpolation}
 [\L^p_{\sigma} (\Omega) , \dom(A_p)]_{\alpha} = \dom(A_p^{\alpha}) \qquad (\alpha \in (0 , 1)),
\end{align}
see~\cite[Thm.~6.6.9]{Haase}. This property will be crucial in the following section.

\section{The square root of $A_p$ and mapping properties of the semigroup}
\label{Sec: Square root}

With the considerations of Section~\ref{Sec: Notation} we can prove Theorem~\ref{Thm: Square root}. Note that this proof is motivated by a calculation Shen performed in~\cite[Lem.~3.5]{Shen-Riesz_transform}.

\begin{proof}[\upshape \bfseries Proof of Theorem~\ref{Thm: Square root}]
To obtain the embedding $\W^{1 , p}_{0 , \sigma} (\Omega) \subset \dom(A_p^{1 / 2})$ combine~\eqref{Eq: Complex interpolation} together with Lemma~\ref{Lem: Embedding} to get the continuous embedding
\begin{align*}
 [\L^p_{\sigma} (\Omega) , \W^{2 , p}_{0 , \sigma} (\Omega)]_{1 / 2} \subset [\L^p_{\sigma} (\Omega) , \dom(A_p)]_{1 / 2} = \dom(A_p^{1 / 2}).
\end{align*}
Now, the desired embedding follows since the interpolation space on the left-hand side is known~\cite[Prop.~2.10, Thm.~2.12]{Mitrea_Monniaux} to be
\begin{align*}
 [\L^p_{\sigma} (\Omega) , \W^{2 , p}_{0 , \sigma} (\Omega)]_{1 / 2} = \W^{1 , p}_{0 , \sigma} (\Omega).
\end{align*}
\indent To obtain the opposite embedding, let $F \in \C_c^{\infty} (\Omega ; \IC^{d \times d})$ and let $p^{\prime}$ be the H\"older conjugate exponent to $p$. Consider the Stokes problem
\begin{align*}
\left\{ \begin{aligned}
 - \Delta u + \nabla \pi &= \IP_{p^{\prime}} \divergence(F) \quad &&\text{in } \Omega \\
 \divergence(u) &= 0 &&\text{in } \Omega \\
 u &= 0 &&\text{on } \partial \Omega.
\end{aligned} \right.
\end{align*}
Because $(\IP_{p^{\prime}} - \Id) \divergence (F) = \nabla g$ for some $g \in \W^{1 , p^{\prime}} (\Omega)$ by Theorem~\ref{Thm: Fabes_Mendez_Mitrea}, the functions $u$ and $\pi - g$ solve the Stokes problem with right-hand side being $\divergence(F)$. Moreover, since $\divergence(F)$ induces a functional in $(\W^{1 , p}_0 (\Omega ; \IC^d))^*$ obeying the estimate
\begin{align*}
 \| \divergence(F) \|_{(\W^{1 , p}_0)^*} \leq \| F \|_{\L^{p^{\prime}}(\Omega ; \IC^{d \times d})},
\end{align*}
we find by~\cite[Thm.~10.6.2]{Mitrea_Wright} that
\begin{align}
\label{Eq: Gradient of the weak solution}
 \| \nabla A_{p^{\prime}}^{-1} \IP_{p^{\prime}} \divergence(F) \|_{\L^{p^{\prime}} (\Omega ; \IC^{d^2})} = \| \nabla u \|_{\L^{p^{\prime}} (\Omega ; \IC^{d^2})} \leq C \| F \|_{\L^{p^{\prime}}(\Omega ; \IC^{d \times d})}.
\end{align}
Here, the constant $C > 0$ is independent of $F$. Appealing to the first part of this proof and to Poincar\'e's inequality delivers
\begin{align}
\label{Eq: Dual gradient estimate}
\begin{aligned}
 \| A_{p^{\prime}}^{- 1 / 2} \IP_{p^{\prime}} \divergence (F) \|_{\L^{p^{\prime}}_{\sigma} (\Omega)} &= \| A_{p^{\prime}}^{1 / 2} A_{p^{\prime}}^{- 1} \IP_{p^{\prime}} \divergence (F) \|_{\L^{p^{\prime}}_{\sigma} (\Omega)} \\
 &\leq C \| \nabla A_{p^{\prime}}^{-1} \IP_{p^{\prime}} \divergence(F) \|_{\L^{p^{\prime}} (\Omega ; \IC^{d^2})} \\
 &\leq C \| F \|_{\L^{p^{\prime}} (\Omega ; \IC^{d \times d})}.
\end{aligned}
\end{align}
Notice that if $\iota : \L^p_{\sigma} (\Omega) \to \L^p (\Omega ; \IC^d)$ denotes the canonical embedding of $\L^p_{\sigma} (\Omega)$ into $\L^p (\Omega ; \IC^d)$, then $\iota^* = \IP_{p^{\prime}}$ and $\nabla A_p^{- 1 / 2} = \nabla \iota A_p^{- 1 / 2}$. Thus, the continuous embedding $\dom(A_p^{1 / 2}) \subset \W^{1 , p}_{0 , \sigma} (\Omega)$ follows by~\eqref{Eq: Dual gradient estimate} by duality.
\end{proof}

\begin{remark}
In the three-dimensional case,~\eqref{Eq: Gradient of the weak solution} is known due to Brown and Shen~\cite[Thm.~2.9]{Brown_Shen} and if $d \geq 3$ it was also proved by Geng and Kilty~\cite[Thm.~1.3]{Geng_Kilty} in the case where $\partial \Omega$ is connected. If $\Omega$ is a bounded and smooth domain, then~\eqref{Eq: Dual gradient estimate} is due to Giga and Miyakawa~\cite[Lem.~2.1]{Giga_Miyakawa}.
\end{remark}

Having determined the domain of the square root of the Stokes operator, the proof of Corollary~\ref{Cor: Lp-Lq} follows immediately from the parabolic smoothing estimate $\| A^{1 / 2} \e^{- t A} x \|_X \leq C t^{- 1 / 2} \| x \|_X$ which is valid for all generators $-A$ of bounded analytic semigroups on a Banach space $X$, see~\cite[Prop.~3.4.3]{Haase}. The $\L^p$-$\L^q$-estimates of the Stokes semigroup are then a straightforward consequence of the gradient estimates and Gagliardo--Nirenberg's inequality~\cite[p.~125]{Nirenberg}.

\section{Existence theory to the Navier--Stokes equations}

\label{Sec: Nonlinear}

When dealing with the Navier--Stokes equations, the space dimension $d$ will be assumed to be $3$ in this section.

\subsection{Solvability in the critical space $\L^{\infty} (0 , T ; \L^3_{\sigma} (\Omega))$ via an iteration scheme}

In this subsection, we discuss the solvability of the Navier--Stokes equations in the mild sense. By this, we mean that $u : [0 , T) \to \W^{1 , r}_{0 , \sigma} (\Omega)$ for some $r \geq 3$ is a continuous function that solves the variation of constants formula 
\begin{align*}
 u (t) = \e^{- t A_r} a - \int_0^t \e^{- (t - s) A_{r / 2}} \IP_{r / 2} (u (s) \cdot \nabla) u (s) \; \d s.
\end{align*}
To obtain such a solution it is a standard strategy to follow the procedure of Giga~\cite[Thm.~4]{Giga_iteration} and Kato~\cite[Thm.~1, Thm.~2]{Kato_iteration}. In this procedure, one defines a successive approximation by
\begin{align}
 u_0 (t) &:=  \e^{- t A_r} a \label{Eq: Initial approximation} \\
 u_{j + 1} (t) &:= u_0 (t) - \int_0^t \e^{- (t - s) A_{r / 2}} \IP_{r / 2} (u_j (s) \cdot \nabla) u_j (s) \; \d s \qquad (j \in \IN). \label{Eq: Nonlinear approximation step}
\end{align}
It is well-known that this sequence converges to a mild solution $u$ whenever the corresponding semigroup operators satisfy
\begin{itemize}
 \item ${\displaystyle \| \e^{- t A_p} f \|_{\L^q_{\sigma} (\Omega)} \leq C t^{- \frac{3}{2} (\frac{1}{p} - \frac{1}{q})} \| f \|_{\L^p_{\sigma} (\Omega)} \quad \text{for all } t > 0 \text{ and } f \in \L^p_{\sigma} (\Omega)}$; \label{Item: Lp-Lq}
 \item ${\displaystyle \| \nabla \e^{- t A_p} f \|_{\L^p(\Omega ; \IC^9)} \leq C t^{- \frac{1}{2}} \| f \|_{\L^p_{\sigma} (\Omega)} \quad \text{for all } t > 0 \text{ and } f \in \L^p_{\sigma} (\Omega)}$, \label{Item: Gradient estimates}
\end{itemize}
for all $1 < p \leq q < \infty$ and a constant $C > 0$ independent of $t$ and $f$. A closer look onto the proofs of Giga and Kato reveals that the estimates in Corollary~\ref{Cor: Lp-Lq} suffice to prove Theorem~\ref{Thm: Critical space}. As the details of this proof are standard, we omit the proof of Theorem~\ref{Thm: Critical space}.

\subsection{An approach via maximal $\L^q$-regularity}
\label{Subsec: An approach via maximal Lq-regularity}

Recall the space $\IE$ defined in~\eqref{Eq: Spaces of maximal regularity}. We say that $u : (0 , \infty) \to \L^p_{\sigma} (\Omega)$ is a strong solution to $\NSE$ if $u \in \IE$, $u$ attains the initial condition in the sense of traces, and there exists $\pi : (0 , \infty) \to \L^p(\Omega)$ such that
\begin{align*}
 \int_{\Omega} \partial_t u (t) \cdot \overline{w} \; \d x + \int_{\Omega} \nabla u (t) \cdot \overline{\nabla w} \; \d x + \int_{\Omega} (u(t) \cdot \nabla) u(t) \cdot \overline{w} \; \d x - \int_{\Omega} \pi \overline{\divergence(w)} \; \d x = \int_{\Omega} f \cdot \overline{w} \; \d x
\end{align*}
holds for every $w \in \W^{1 , p^{\prime}}_0 (\Omega ; \IC^3)$ and almost every $t > 0$, where $1 / p + 1 / p^{\prime} = 1$. \par
In order to derive the existence of solutions to the Navier--Stokes equations via maximal $\L^q$-regularity one usually performs the following steps:
\begin{enumerate}
\item Recast the Navier--Stokes equations on the subspace $\L^p_{\sigma} (\Omega)$ as
\begin{align*}
 \mathrm{(\IP NSE)} \quad \left\{ \begin{aligned}
 \partial_t u + A_p u &= \IP_p f - \IP_p ( u \cdot \nabla ) u \\
 u(0) &= a.
\end{aligned} \right.
\end{align*}
\item Replace the term $\IP_p ( u \cdot \nabla ) u$ by $\IP_p ( v \cdot \nabla ) v$, with $v \in \IE$, and show that for all $v \in \IE$ we have $\IP_p ( v \cdot \nabla ) v \in \L^q(0 , \infty ; \L^p_{\sigma} (\Omega))$. For fixed $v \in \IE$ the maximal $\L^q$-regularity of $A_p$ provides then a unique solution $u_v \in \IE$ to the corresponding linear problem. \label{It: Replaced nonlinearity}
\item For $f$ and $a$ fixed, show that the linear operator mapping $v$ to $u_v$ has a fixed point.
\end{enumerate}

\begin{remark}
To prove the existence of strong solutions to the Navier--Stokes equations this approach is quite standard. However, to assure~\eqref{It: Replaced nonlinearity} it is often necessary to choose $p$ and $q$ in the definition of $\IE$ very large, which makes certain embedding theorems available. On a bounded Lipschitz domain, however, the domain of the Stokes operator lacks to embed into a Sobolev space of second-order and there is the additional restriction of having the semigroup theory only available for $p$ in the interval about $2$ given in~\eqref{Eq: Admissible interval of p's}. That it is still possible to choose $p$ and $q$ properly in the three-dimensional case is presented in the remainder of this article.
\end{remark}

To verify~\eqref{It: Replaced nonlinearity} it is essential to have good embeddings of $\IE$ into a space of the form
\begin{align*}
 \L^r(0 , \infty ; \W^{1 , p} (\Omega ; \IC^3)),
\end{align*}
for some suitable $r > q$, which is desired to be as large as possible. To obtain this embedding, the following two results are of great importance. The first result deals with embedding properties of the domain of the Stokes operator and is in the case $p = 2$ due to Brown and Shen~\cite[Thm.~2.12]{Brown_Shen}, see also Mitrea and Monniaux~\cite[Thm.~5.3]{Mitrea_Monniaux}, and in the case $p > 2$ due to Mitrea and Wright~\cite[Thm.~10.6.2]{Mitrea_Wright}. In the following, the Bessel potential spaces will be denoted by $\H^{s , p}$.

\begin{theorem}[Brown \& Shen, Mitrea \& Wright]
\label{Thm: Embedding domains}
Let $\Omega \subset \IR^3$ be a bounded Lipschitz domain.
\begin{enumerate}
 \item In the case $p = 2$ the continuous embedding $\dom(A_2) \subset \H^{3 / 2 , 2} (\Omega ; \IC^3)$ is valid.
 \item There exists $\delta \in (0 , 1]$ such that for all $2 < p < 3 + \delta / (1 - \delta)$ and all $s \in [0 , 1)$ the continuous embedding $\dom(A_p) \subset \H^{s + 1 / p , p} (\Omega ; \IC^3)$ holds.
\end{enumerate}
\end{theorem}

The second result that is needed is a consequence of the mixed derivative theorem. The version presented here is essentially due to Denk and Kaip~\cite[Lem.~2.61]{Denk_Kaip} and is extended here to non-vanishing functions at $t = 0$ by a reflection argument.

\begin{proposition}
\label{Prop: Mixed derivative theorem}
Let $\Omega \subset \IR^d$ be a bounded Lipschitz domain and $1 < p , q < \infty$. Then, for every $s \geq 0$ and $\sigma \in [0 , 1]$ the continuous embedding
\begin{align*}
 \W^{1 , q} (0 , \infty ; \L^p(\Omega)) \cap \L^q(0 , \infty ; \H^{s , p} (\Omega)) \subset \H^{\sigma , q} (0 , \infty ; \H^{(1 - \sigma) s , p} (\Omega))
\end{align*}
holds.
\end{proposition}

\begin{proof}
Let $\Xi \subset \IR^d$. If $\Xi = \IR^d$, then by~\cite[Lem.~2.61]{Denk_Kaip} the following continuous embedding
\begin{align}
\label{Eq: Embedding for zero time trace}
 \W^{1 , q}_0 (0 , \infty ; \L^p(\Xi)) \cap \L^q(0 , \infty ; \H^{s , p} (\Xi)) \subset \H^{\sigma , q} (0 , \infty ; \H^{(1 - \sigma) s , p} (\Xi))
\end{align}
is valid. To obtain~\eqref{Eq: Embedding for zero time trace} for $\Xi = \Omega$, extend functions on $\Omega$ to all of $\IR^d$ by employing Stein's extension operator~\cite[Thm.~VI.3.5]{Stein} and apply~\eqref{Eq: Embedding for zero time trace} in the case $\Xi = \IR^d$ to this extended function. Next, for general
\begin{align*}
 u \in \W^{1 , q} (0 , \infty ; \L^p(\Omega)) \cap \L^q(0 , \infty ; \H^{s , p} (\Omega)),
\end{align*}
extend $u$ to a function $\widetilde{u}$ on $\IR$ by an even reflection and multiply the extended function by a smooth cut-off function $\varphi$ that is one on $[0 , \infty)$ and zero on $(- \infty , -1]$. Finally, the shifted function $[\widetilde{u} \varphi] (\cdot - 2)$ lies in the set on the left-hand side of~\eqref{Eq: Embedding for zero time trace} with $\Xi = \Omega$ and the corresponding continuous embedding implies the statement of the proposition.
\end{proof}

\begin{proposition}
\label{Prop: Maximal regularity embedding}
Let $\Omega \subset \IR^3$ be a bounded Lipschitz domain, $\delta \in (0 , 1]$ as in Theorem~\ref{Thm: Embedding domains}, and $2 \leq p < 3 + \delta / (1 - \delta)$. Then, for
\begin{align*}
 1 < s < 1 + \frac{1}{p}, \quad \text{if} \quad p > 2 \qquad \text{and} \qquad 1 < s \leq \frac{3}{2}, \quad \text{if} \quad p = 2
\end{align*}
and $1 < q < s / (s - 1)$ the following continuous embedding holds
\begin{align*}
\IE \subset \L^{\frac{s q}{s - s q + q}} (0 , \infty ; \W^{1 , p} (\Omega ; \IC^3)).
\end{align*}
\end{proposition}

\begin{proof}
The proposition readily follows by combining Theorem~\ref{Thm: Embedding domains} and Proposition~\ref{Prop: Mixed derivative theorem} together with Sobolev's embedding theorem.
\end{proof}

Having a suitable embedding of $\IE$ at hand, we can start to estimate the nonlinear term. To do so, the following theorem of Brown and Shen~\cite[Thm.~3.1]{Brown_Shen} is the final ingredient.

\begin{theorem}[Brown \& Shen]
\label{Thm: Brown_Shen}
Let $\Omega \subset \IR^3$ be a bounded Lipschitz domain. Then there exists a constant $C > 0$ depending only on the Lipschitz character of $\Omega$ such that
\begin{align*}
 \| u \|_{\L^{\infty} (\Omega ; \IC^3)} \leq C \| \nabla u \|_{\L^2(\Omega ; \IC^9)}^{1 / 2} \| A_2 u \|_{\L^2_{\sigma}(\Omega)}^{1 / 2} \qquad (u \in \dom(A_2)).
\end{align*}
\end{theorem}

The following lemma gives the estimate of the nonlinear term and thereby concludes Step~\eqref{It: Replaced nonlinearity} of our three steps agenda.

\begin{lemma}
\label{Lem: Right-hand side in the right space}
Let $\Omega \subset \IR^3$ be a bounded Lipschitz domain, $\delta \in (0 , 1]$ be as in Theorem~\ref{Thm: Embedding domains}, and $2 \leq p < 3 + \delta / (1 - \delta)$. In the case $p = 2$, let $2 \leq q < \infty$, and in the case $p > 2$, let
\begin{align*}
 \frac{2 (p + 1)}{3} < q < \infty.
\end{align*}
Then, there exists $C > 0$ such that for all $v , w \in \IE$
\begin{align*}
 \| ( v \cdot \nabla ) w\|_{\L^q(0 , \infty ; \L^p(\Omega ; \IC^3))} \leq C \|v\|_{\IE} \|w\|_{\IE}.
\end{align*}
\end{lemma}

\begin{proof}
Let $v , w \in \IE$. Since $p \geq 2$ we find $\dom(A_p) \subset \dom(A_2)$ so that $v(t) \in \dom(A_2)$ for almost every $t > 0$. By Theorem~\ref{Thm: Brown_Shen} there exists a constant $C > 0$ such that
\begin{align*}
 \|v(t)\|_{\L^{\infty} (\Omega ; \IC^3)} \leq C \|\nabla v(t)\|_{\L^2(\Omega ; \IC^9)}^{1 / 2} \|A_2 v(t)\|_{\L^2_{\sigma}(\Omega)}^{1 / 2} \qquad (\text{a.e.\@ } t > 0).
\end{align*}
Thus, there exists a constant $C > 0$ such that
\begin{align*}
 \int_0^{\infty} \| ( v(t) \cdot \nabla ) w(t)\|_{\L^p(\Omega ; \IC^3)}^q \; \d t &\leq C \int_0^{\infty} \|\nabla v(t)\|_{\L^2(\Omega ; \IC^9)}^{q / 2} \|A_2 v(t)\|_{\L^2_{\sigma}(\Omega)}^{q / 2} \|\nabla w(t)\|_{\L^p(\Omega ; \IC^9)}^q \; \d t.
 \intertext{An application of H\"older's inequality in space and time shows}
 &\leq C \bigg( \int_0^{\infty} \|\nabla v(t)\|_{\L^p(\Omega ; \IC^9)}^{3q} \; \d t \bigg)^{\frac{1}{6}} \bigg( \int_0^{\infty} \|\nabla w(t)\|_{\L^p(\Omega ; \IC^9)}^{3q}  \; \d t \bigg)^{\frac{1}{3}} \\
 &\qquad \cdot \bigg( \int_0^{\infty} \|A_p v(t)\|_{\L^p_{\sigma}(\Omega)}^q \; \d t \bigg)^{\frac{1}{2}}.
\end{align*}
Finally, we would like to appeal to Proposition~\ref{Prop: Maximal regularity embedding}. For this purpose, we have to ensure that there exists a number $s$ subject to the premises given in the very same proposition, such that
\begin{align*}
 1 < q < \frac{s}{s - 1} \quad \text{and} \quad 3q = \frac{s q}{s - s q + q}.
\end{align*}
One readily verifies that $s$ given by $s := 3 q / (3 q - 2)$ meets these requirements. Now, we can use Proposition~\ref{Prop: Maximal regularity embedding} to estimate
\begin{align*}
 \bigg( \int_0^{\infty} \|\nabla v(t)\|_{\L^p(\Omega ; \IC^9)}^{3q} \; \d t \bigg)^{\frac{1}{6}} \bigg( \int_0^{\infty} \|\nabla w(t)\|_{\L^p(\Omega ; \IC^9)}^{3q}  \; \d t \bigg)^{\frac{1}{3}} \leq C \| v \|_{\IE}^{q / 2} \| w \|_{\IE}^q.& \qedhere
\end{align*}
\end{proof}

Let $p$ and $q$ be as in Lemma~\ref{Lem: Right-hand side in the right space} and fix $f \in \L^q(0 , \infty ; \L^p (\Omega ; \IC^3))$ and $a \in (\L^p_{\sigma} (\Omega) ; \dom(A_p))_{1 - 1 / q , q}$. Let $\cT_{(f , a)} : \IE \to \IE$ be the mapping that maps $v$ to $u_v$, where $u_v$ is given by
\begin{align*}
 \left\{ \begin{aligned}
  \partial_t u_v + A_p u_v &= \IP_p f - \IP_p (v \cdot \nabla) v \qquad t > 0 \\
  u_v(0) &= a.
 \end{aligned} \right.
\end{align*}
By Lemma~\ref{Lem: Right-hand side in the right space} the right-hand side lies in $\L^q(0 , \infty ; \L^p_{\sigma} (\Omega))$, so that the maximal $\L^q$-regularity of $A_p$, see Theorem~\ref{Thm: Maximal regularity}, implies that the solution $u_v$ indeed exists and lies in $\IE$. It follows that for each $f$ and $a$ chosen in the spaces above, the mapping $\cT_{(f , a)}$ is well-defined. The following proposition shows that $\cT_{(f , a)}$ has a unique fixed point, provided $f$ and $a$ are small enough. The proof follows exactly the lines of Saal~\cite[Thm.~1.2]{Saal} and is thus omitted.

\begin{proposition}
\label{Prop: Existence of the fixed point}
Let $\Omega \subset \IR^3$ be a bounded Lipschitz domain, $\eps > 0$ be the minimal $\eps$ of Theorems~\ref{Thm: Fabes_Mendez_Mitrea} and~\ref{Thm: Maximal regularity}, $\delta \in (0 , 1]$ as in Theorem~\ref{Thm: Embedding domains}, and $2 \leq p < 3 + \min\{\delta / (1 - \delta) , \eps\}$. In the case $p = 2$, let $2 \leq q < \infty$, and in the case $p > 2$, let
\begin{align*}
 \frac{2 (p + 1)}{3} < q < \infty.
\end{align*}
Then, there exists a constant $C > 0$, such that for all
\begin{align*}
 (f , a) \in \L^q (0 , \infty ; \L^p (\Omega ; \IC^3)) \times (\L^p_{\sigma} (\Omega) , \dom(A_p))_{1 - 1 / q , q}
\end{align*}
with
\begin{align*}
 \| a \|_{(\L^p_{\sigma} (\Omega) , \dom(A_p))_{1 - 1 / q , q}} + \| f \|_{\L^q(0 , \infty ; \L^p (\Omega ; \IC^3))} < C,
\end{align*}
there exists a unique fixed point $u \in \IE$ of $\cT_{(f , a)}$.
\end{proposition}

The proof of Theorem~\ref{Thm: Solutions via maximal regularity} is now a mere reformulation of Proposition~\ref{Prop: Existence of the fixed point} and is thus omitted.

\begin{remark}
In contrast to the result of Taylor~\cite[Prop.~3.1]{Taylor}, Proposition~\ref{Prop: Existence of the fixed point} allows us to construct not only mild, but \textit{strong} solutions to the Navier--Stokes equations in bounded Lipschitz domains, that also satisfy Serrin's uniqueness condition. Mitrea and Monniaux~\cite{Mitrea_Monniaux} construct mild solutions with no external force and initial conditions in $\dom(A_2^{1 / 4})$. They also prove~\cite[Thm.~6.4]{Mitrea_Monniaux} that this mild solution is a strong solution in the class
\begin{align*}
 \W^{1 , q} (0 , T ; \L^2_{\sigma} (\Omega)) \cap \L^q (0 , T ; \dom(A_2)) \qquad (1 < q < \tfrac{4}{3}).
\end{align*}
Thus, Proposition~\ref{Prop: Existence of the fixed point} can be regarded as a counterpart of this result in classes of high regularity.
\end{remark}



\end{document}